\documentclass[letterpaper,11pt]{article}

\usepackage[margin=1in]{geometry}  % set the margins to 1in on all sides
\usepackage{xspace,exscale,relsize}
\usepackage{fancybox,shadow}
\usepackage{graphicx}
\usepackage{color}
\usepackage{amsthm,amsfonts}
\usepackage{amssymb}
\usepackage{authblk}
\usepackage[title]{appendix}

\usepackage{extarrows}

\usepackage[noadjust]{cite}

\usepackage{amsmath}
	\makeatletter
	\let\over=\@@over \let\overwithdelims=\@@overwithdelims
	\let\atop=\@@atop \let\atopwithdelims=\@@atopwithdelims
  	\let\above=\@@above \let\abovewithdelims=\@@abovewithdelims
  	\makeatother
\usepackage{rotating}

\usepackage{ifpdf}

\usepackage{subfigure}
\usepackage{psfrag}

\usepackage{prettyref}
\usepackage{enumitem}

\usepackage{tikz}
\usetikzlibrary{arrows}
\tikzstyle{int}=[draw, fill=blue!20, minimum size=2em]
\tikzstyle{dot}=[circle, draw, fill=blue!20, minimum size=2em]
\tikzstyle{init} = [pin edge={to-,thin,black}]

\usepackage[%dvips,
            CJKbookmarks=true,
            bookmarksnumbered=true,
            bookmarksopen=true,
            colorlinks=true,
            citecolor=red,
            linkcolor=blue,
            anchorcolor=red,
            urlcolor=blue,
            pdfauthor={Polyanskiy and Wu}
            ]{hyperref}

\usepackage[all]{xy}

\usepackage{mathtools}
\usepackage{ifthen}
\newboolean{aos}
\setboolean{aos}{FALSE}

\newcommand{\mreals}{\ensuremath{\mathbb{R}}}

\ifx\eqref\undefined
	\newcommand{\eqref}[1]{~(\ref{#1})}
\fi
\ifx\mod\undefined
	\def\mod{\mathop{\rm mod}}
\fi

\usepackage{bm}

\def\EE{\Expect}

\def\PP{\mathbb{P}}

\def\eqdef{\triangleq}

\newcommand{\reals}{\mathbb{R}}

\newcommand{\complex}{\mathbb{C}}
\newcommand{\Expect}{\mathbb{E}}

\definecolor{myblue}{rgb}{.8, .8, 1}
\definecolor{mathblue}{rgb}{0.2472, 0.24, 0.6} % mathematica's Color[1, 1--3]
\definecolor{mathred}{rgb}{0.6, 0.24, 0.442893}
\definecolor{mathyellow}{rgb}{0.6, 0.547014, 0.24}

\newcommand{\calN}{{\mathcal{N}}}

\newrefformat{eq}{(\ref{#1})}
\newrefformat{thm}{Theorem~\ref{#1}}
\newrefformat{th}{Theorem~\ref{#1}}
\newrefformat{chap}{Chapter~\ref{#1}}
\newrefformat{sec}{Section~\ref{#1}}
\newrefformat{seca}{Section~\ref{#1}}
\newrefformat{algo}{Algorithm~\ref{#1}}
\newrefformat{fig}{Fig.~\ref{#1}}
\newrefformat{tab}{Table~\ref{#1}}
\newrefformat{rmk}{Remark~\ref{#1}}
\newrefformat{clm}{Claim~\ref{#1}}
\newrefformat{def}{Definition~\ref{#1}}
\newrefformat{cor}{Corollary~\ref{#1}}
\newrefformat{lmm}{Lemma~\ref{#1}}
\newrefformat{prop}{Proposition~\ref{#1}}
\newrefformat{pr}{Proposition~\ref{#1}}
\newrefformat{app}{Appendix~\ref{#1}}
\newrefformat{apx}{Appendix~\ref{#1}}
\newrefformat{ex}{Example~\ref{#1}}
\newrefformat{exer}{Exercise~\ref{#1}}
\newrefformat{soln}{Solution~\ref{#1}}
\def\unifto{\mathop{{\mskip 3mu plus 2mu minus 1mu%
	\setbox0=\hbox{$\mathchar"3221$}%
	\raise.6ex\copy0\kern-\wd0%
	\lower0.5ex\hbox{$\mathchar"3221$}}\mskip 3mu plus 2mu minus 1mu}}

\ifx\lesssim\undefined
\def\simleq{{{\mskip 3mu plus 2mu minus 1mu%
	\setbox0=\hbox{$\mathchar"013C$}%
	\raise.2ex\copy0\kern-\wd0%
	\lower0.9ex\hbox{$\mathchar"0218$}}\mskip 3mu plus 2mu minus 1mu}}
\else
\def\simleq{\lesssim}
\fi

\ifx\gtrsim\undefined
\def\simgeq{{{\mskip 3mu plus 2mu minus 1mu%
	\setbox0=\hbox{$\mathchar"013E$}%
	\raise.2ex\copy0\kern-\wd0%
	\lower0.9ex\hbox{$\mathchar"0218$}}\mskip 3mu plus 2mu minus 1mu}}
\else
\def\simgeq{\gtrsim}
\fi

\newtheorem{theorem}{Theorem}

\newtheorem{corollary}[theorem]{Corollary}

\newtheorem{proposition}[theorem]{Proposition}

\theoremstyle{definition}

\newtheorem{remark}{Remark}

\newif\ifmapx
{\catcode`/=0 \catcode`\\=12/gdef/mkillslash\#1{#1}}
\edef\jobnametmp{\expandafter\string\csname gaussian_approx_apx\endcsname}
\edef\jobnameapx{\expandafter\mkillslash\jobnametmp}
\edef\jobnameexpand{\jobname}
\ifx\jobnameexpand\jobnameapx
\mapxtrue
\else
\mapxfalse
\fi

\renewcommand{\tilde}{\widetilde}

\begin{document}
\ifpdf
\DeclareGraphicsExtensions{.pgf}
\graphicspath{{figures/}{plots/}}
\fi

% paper title
\title{Note on approximating the Laplace transform of a Gaussian on a complex disk}

\author{Yury Polyanskiy and Yihong Wu\thanks{Y.P. is with the Department of EECS, MIT, Cambridge, MA, email: \url{yp@mit.edu}. Y.W. is with
the Department of Statistics and Data Science, Yale University, New Haven, CT, email: \url{yihong.wu@yale.edu}.}}

\maketitle

\begin{abstract}
%Consider the standard normal distribution $\pi=\mathcal{N}(0,1)$. In this short note we ask how can it be best approximated by
%a distribution supported on $[-a,a]$? Perhaps, the natural conjecture is that for large $a$ the almost optimal choice is
%given by truncating $\pi$ to $[-a,a]$. Indeed, such approximation achieves the optimal rate of $e^{-\Theta(a^2)}$ in
%terms of the $L_\infty$-distance between characteristic functions. However, if we consider the $L_\infty$-distance between
%Laplace transforms on a unit disk of $\mathbb{C}$, the optimal rate is $e^{-\Theta(a^2 \log a)}$, while truncation still
%only attains $e^{-\Theta(a^2)}$. The optimal rate is attained by the Gauss-Hermite quadrature. 

In this short note we study how well a Gaussian distribution can be approximated by distributions supported on $[-a,a]$.
Perhaps, the natural conjecture is that for large $a$ the almost optimal choice is
given by truncating the Gaussian to $[-a,a]$. Indeed, such approximation achieves the optimal rate of $e^{-\Theta(a^2)}$ in
terms of the $L_\infty$-distance between characteristic functions. However, if we consider the $L_\infty$-distance between
Laplace transforms on a complex disk, the optimal rate is $e^{-\Theta(a^2 \log a)}$, while truncation still
only attains $e^{-\Theta(a^2)}$. The optimal rate can be attained by the Gauss-Hermite quadrature. As corollary, we also
construct a ``super-flat'' Gaussian mixture of $\Theta(a^2)$ components with means in $[-a,a]$ and
whose density has all derivatives bounded by $e^{-\Omega(a^2 \log(a))}$ in the $O(1)$-neighborhood of the origin.
\end{abstract}

\section{Approximating the Gaussian}

We study the best approximation of a Gaussian distribution by compact support measures, in the
sense of the uniform approximation of the Laplace transform on a complex disk.
Let $L_\pi(z) = \int_\mreals d\pi(y) e^{zy}$ be the Laplace transform, $z \in \mathbb{C}$, of the
measure $\pi$ and $\Psi_\pi(t) \eqdef L_{\pi}(it)$ be its characteristic function. Denote $L_0(z) = e^{z^2/2}$ and
$\Psi_0(t) = e^{-t^2/2}$ the Laplace transform and the characteristic function corresponding to the standard Gaussian $\pi_0 =
\mathcal{N}(0,1)$ with density
$$ \phi(x) \eqdef {1\over \sqrt{2\pi}} e^{-x^2/2}\,, \quad x \in \mreals\,.$$

How well can a measure $\pi_1$ with support on $[-a,a]$ approximate $\pi_0$? Perhaps the most natural
choice for $\pi_1$ is the truncated $\pi_0$:
\begin{equation}\label{eq:pi_trunc}
	\pi_1(dx) \eqdef \phi_a(x) dx, \qquad \phi_a(x) \triangleq {\phi(x) \over 1-2Q(a)} 1\{|x| \le a\}\,,
\end{equation}
where $Q(a) = \PP[\mathcal{N}(0,1) > a]$.
Indeed, truncation is asymptotically optimal (as $a\to\infty$) 
in approximating the characteristic function, as made preicse by the following result:

\begin{proposition} There exists some $c>0$ such that for all $a\ge1$ and any probability measure $\pi_1$ supported on $[-a,a]$ we have
	\begin{equation}\label{eq:cf_0}
		\sup_{t\in \mreals} |\Psi_{\pi_1}(t) - e^{-t^2/2}| \ge c e^{-c a^2}\,.
	\end{equation}	
	Furthermore, truncation~\eqref{eq:pi_trunc} satisfies (for $a\ge 1$)
	\begin{equation}\label{eq:cf_1}
		\sup_{t\in \mreals} |\Psi_{\pi_1}(t) - e^{-t^2/2}| \le 2 e^{-a^2/2}\,.
	\end{equation}	
\end{proposition}
\begin{proof}
Let us define $B(z) \eqdef L_{\pi}(z) - e^{z^2/2}$, a holomorphic (entire) function on $\mathbb{C}$. Note that if
$\Re(z)=r$ then
	$$ |L_\pi(z)| \le e^{a|r|}, \qquad |e^{z^2/2}| \le e^{r^2/2}\,,$$
and thus for $r\ge 2a$
	\begin{equation}\label{eq:cf_2}
		b(r) \eqdef \sup\{|B(z)|: \Re(z) = r\} \le e^{ar} + e^{r^2/2} \le 2e^{r^2/2} 
\end{equation}	
On the other hand, for every $r \ge 3a, a\ge1$ we have 
	\begin{equation}\label{eq:cf_3}
		b(r) \ge |B(r)| \ge  e^{r^2/2} - e^{ar} \ge {1\over 2} e^{r^2/2}\,.
\end{equation}	
Applying the Hadamard three-lines theorem to $B(z)$, we conclude that $r \mapsto \log b(r)$ is convex and hence
	\begin{equation}\label{eq:cf_4}
		b(3a) \le (b(0))^{1/2} (b(6a))^{1/2}\,.
	\end{equation}	
Since the left-hand side of~\eqref{eq:cf_0} equals $b(0)$, \eqref{eq:cf_0} then follows from \eqref{eq:cf_2}-\eqref{eq:cf_4}.

For the converse part, in view of \eqref{eq:pi_trunc}, the total variation between $\pi_0$ and its conditional version $\pi_1$ is given by 
	$$ \int_{\mreals} |\phi_a(x) - \phi(x)| dx = 
	2 \|\pi_0-\pi_1\|_{\rm TV}= 2 \pi_0([-a,a]^c) = 	4Q(a)\,$$
	Therefore for the Fourier transform of $\phi_a-\phi$ we get
	$$ \sup_{t\in \mreals} |\Psi_{\pi_1}(t) - e^{-t^2/2}| \le 4Q(a) \le {4\over \sqrt{2\pi}a} e^{-a^2/2} \le
	2e^{-a^2/2}\,.$$
\end{proof}

Despite this evidence, it turns out that for the purpose of approximating Laplace transform in a neighborhood of $0$,
there is a much better approximation than~\eqref{eq:pi_trunc}.

\begin{theorem}\label{th:approx} 
There exists some constant $c>0$ such that for any probability measure $\pi$ supported on $[-a,a]$, $a\ge 1$,  we have
	\begin{equation}\label{eq:ap_lbnd}
		\sup_{|z| \le 1, z\in\mathbb{C}} | e^{z^2/2} - L_{\pi}(z)| \ge c e^{-c a^2 \log(a)}\,, 
\end{equation}	

	Furthermore, there exists an absolute constant $c_1>0$ so that for all $b\ge 1$ and all $a\ge 2c_1b$ there exists
	distribution $\pi$ (the Gauss-Hermite quadrature) supported on $[-a,a]$ such that 
	$$ \sup_{|z| \le b, z\in\complex} |e^{z^2/2} - L_{\pi}(z)| \le 3 (c_1 b/a)^{a^2/4}\,. $$
	Taking $b=1$ implies that the bound~\eqref{eq:ap_lbnd} is order-optimal.
\end{theorem}

\begin{remark}
\label{rmk:truncation}	
	When $\pi_1$ is given by the truncation~\eqref{eq:pi_trunc}, then performing explicit calculation for
	$z\in\reals$ we have
	\[
	L_{\pi_1}(z) = e^{z^2/2} \frac{\Phi(a+z)+\Phi(a-z)-1}{2\Phi(a)-1}\,.
	\]
	The same expression (by analytic continuation) holds for arbitrary $z\in \mathbb{C}$ if $\Phi(z)$ is understood
	as solution of $\Phi'(z) = \phi(z), \Phi(0)=1/2$. For $|z|=O(1)$, the approximation error is $e^{-\Omega(a^2)}$,
	rather than $e^{-\Omega(a^2 \log a)}$. The suboptimality of truncation is demonstrated on Fig.~\ref{fig:plot}.
\end{remark}

\begin{figure}[t]
\centering
\includegraphics[width=.6\textwidth]{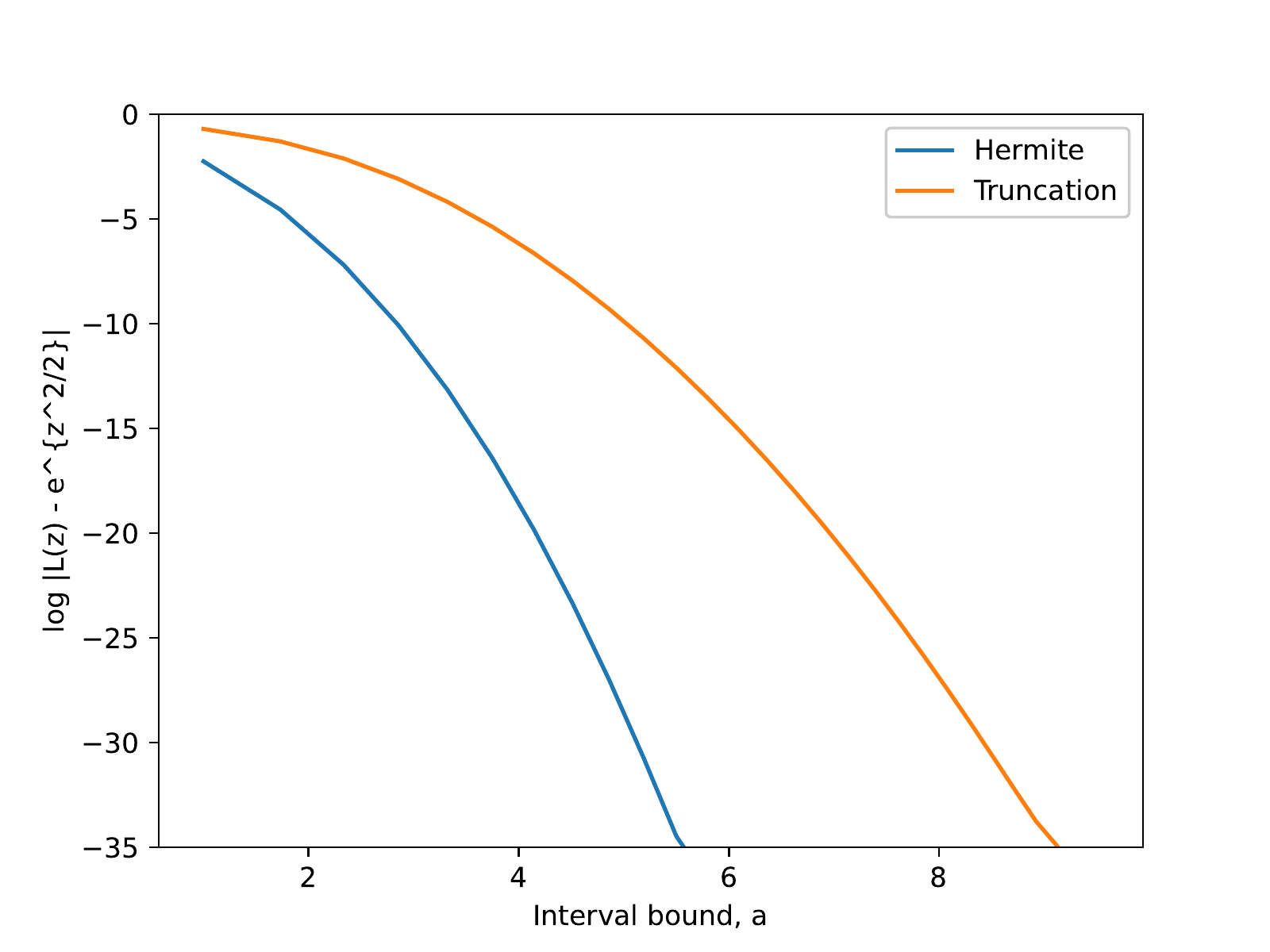}
\caption{Comparison of approximations of $\mathcal{N}(0,1)$ by distributions supported on $[-a,a]$, as measured by the
(log of) $L_\infty$ distance between the Laplace transform on the unit disk in $\mathbb{C}$.}
\label{fig:plot}
\end{figure}

\begin{proof}

	As above, denote $B(z) = L_{\pi}(z) - e^{z^2/2}$, and define 
	$$ M(r) = \sup_{|z| \le r, z\in \mathbb{C}} |B(z)|\,.$$
	From~\eqref{eq:cf_3} we have for any $r\ge 3a, a\ge 1$
		$$ M(r) \ge {1\over 2} e^{r^2/2}$$
	and from $|L_{\pi}(z)| \le e^{a|z|}$ we also have (for any $r\ge 2a$):
		$$ M(r) \le e^{ar} + e^{r^2/2} \le 2 e^{r^2/2}\,.$$
	Applying the Hadamard three-circles theorem, we have $\log r \mapsto \log M(r)$ is convex, and hence
		$$ M(3a) \le (M(1))^{1-\lambda} (M(5a))^{\lambda}\,,$$
		where $\lambda = {\log(5a)\over \log (3a)}$ and $1-\lambda \Theta({1\over \log a})$.
	From here we obtain for some constant $c>0$
		$$ \log M(1) \ge c (-a^2 -1) \log(a)\,,$$
	which proves~\eqref{eq:ap_lbnd}.

	For the upper bound, take $\pi$ to be the $k$-point Gauss-Hermite quadrature of $\mathcal{N}(0,1)$ (cf.~\cite[Section 3.6]{stoer.2002}). This is the unique $k$-atomic distribution that matches the first $2k-1$ moments of $\calN(0,1)$. Specifically, we have:
	\begin{itemize}
	\item  $\pi$ is supported on the roots of the degree-$k$ Hermite polynomial,	which lie in $[-\sqrt{4k+2},\sqrt{4k+2}]$ \cite[Theorem 6.32]{orthogonal.poly};
	\item The $i$-th moment of $\pi$, denoted by $m_i(\pi)$, satisfies $m_i(\pi)=m_i(N(0,1))$ for all $i=1,\ldots,2k-1$.
	\item $\pi$ is symmetric so that all odd moments are zero.
	\end{itemize}
We set $k=\lceil a^2/8 \rceil$, so that $\pi$ is supported on $[-a,a]$.

	Let us denote $X\sim \pi$ and $G \sim \mathcal{N}(0,1)$. By Taylor expansion we get
	$$ B(z) = \EE[e^{zX}] - \EE[e^{zG}] = \sum_{m=2k}^\infty {1\over m!} z^m (\EE[X^m] - \EE[G^m]) = \sum_{\ell = k}^\infty 
		{1\over (2\ell)!} z^{2\ell} (\EE[X^{2\ell}] - \EE[G^{2\ell}])\,.$$
	Now, we will bound $(2\ell)!\ge (2\ell/e)^{2\ell}$, $\EE[X^{2\ell}] \le a^{2\ell}$, 
	$\EE[G^{2\ell}] = (2\ell -1)!! \le (2\ell)^\ell$. This implies that for all $|z|\le b$ we have
	\begin{align*} |B(z)| &\le \sum_{\ell \ge k} \left\{\left(ea\over 2\ell\right)^{2\ell} + \left(e\over
	\sqrt{2\ell}\right)^{2\ell}\right\} |z|^{2\ell}\\
		&\le \sum_{\ell \ge k} \left\{\left(ea\over 2k\right)^{2\ell} + \left(e\over
	\sqrt{2k}\right)^{2\ell}\right\} |z|^{2\ell}\\
		&\le 2 \sum_{\ell \ge k} (c_1b/a)^{2\ell}\,,
	\end{align*}		
	where in the last step we used ${ea\over2k}, {e\over \sqrt{2k}} \le {c_1\over a}$ for some absolute
	constant $c_1$. In all, we have that whenever $c_1b/a<1/2$ we get
	$$ |B(z)| \le {2\over 1- (c_1b/a)^2} (c_1b/a)^{2k} \le 3 (c_1b/a)^{a^2/4}\,.$$

\end{proof}

\begin{remark}
	Note that our proof does not show that for any $\pi$ supported on $[-a,a]$, its characteristic function restricted on $[-1,1]$ must satisfy:
		$$ \sup_{|t|\le 1} |\Psi_\pi(t) - e^{-t^2/2}| \ge c e^{-ca^2 \log a}\,.$$
	It is natural to conjecture that this should hold, though. 

\end{remark}
\begin{remark} Note also that the Gauss-Hermite quadrature considered in the
	theorem, while essentially optimal on complex disks, is not \emph{uniformly} better than the naive truncation.
	For example, due to its finite support, the Gauss-Hermite quadrature is a very bad approximation in the sense
	of~\eqref{eq:cf_0}. Indeed, for any finite discrete distribution $\pi_1$ we have $\limsup_{t\to\infty}
	|\Psi_{\pi_1}(t)|=1$, thus only attaining the trivial bound of $1$ in the right-hand side of~\eqref{eq:cf_0}. 
	(To	see this, note that $\Psi_{\pi_1}(t) = \sum_{j=1}^k p_i e^{it\omega_j}$. By simultaneous rational
	approximation (see, e.g., \cite[Theorem VI, p.~13]{cassels72}), we have infinitely many values $q$ such that for
	all $j$ 
	$|q{\omega_j\over 2\pi} - p_j| < {1\over q^{1/k}}$ for some $p_j \in \mathbb{Z}$. In turn, this implies that 
	$\liminf_{t\to\infty} \max_{j=1,\ldots,m} ( t\omega_j \mod 2\pi) = 0$, and that $\Psi_{\pi_1}(t)\to1$ along the
	subsequence of $t$ attaining $\liminf$.)
	%(To	see this, note that $\Psi_{\pi_1}(t) = \sum_{k=1}^m p_k e^{it\omega_k}$. Now consider the linear flow on the
	%torus given by $t \mapsto (t\omega_1,\ldots,t\omega_m) \mod 2\pi$. It is known that trajectory under
	%this flow is dense in some subtorus~\cite[Exercise 1.4.1]{katok1997introduction}. In any case,
	%the trajectory 	should return infinitely often 
	%\nb{Say ``Thus, the trajectory returns infinitely often ''?}
	%to arbitrarily small
	%neighborhood of $0$ in $[0,2\pi)^m$. For such special values of $t$ we will have $\Psi_{\pi_1}(t) \approx 1$.)
\end{remark}

\section{Super-flat Gaussian mixtures}

As a corollary of construction in the previous section we can also derive a curious discrete distribution $\pi_2 =
\sum_{m} w_m \delta_{x_m}$ supported on $[-a,a]$ such that its convolution with the Gaussian kernel $\pi_2 * \phi$ is maximally flat near the
origin. More precisely, we have the following result.

\begin{corollary} There exist constants $C_1, C>0$ such that for every $a>0$ there exists $k=\Theta(a^2)$, $w_m \ge 0$
with $\sum_{m=1}^k w_m = 1$ and $x_m \in
[-a,a]$, $m\in\{1,\ldots,k\}$, such that
	$$ \left|\left({d\over dz}\right)^n \sum_{m=1}^k w_m \phi(z-x_m)\right| \le n! \cdot C_1 e^{-Ca^2 \log(a)} \qquad
	\forall z\in \mathbb{C}, |z| \le 1, n\in \{1,2,\ldots\}\,.$$
\end{corollary}

\begin{proof} Consider the distribution $\pi = \sum_{m=1}^k \tilde w_m \delta_{x_m}$ claimed by Theorem~\ref{th:approx}
for $b=2$. Then (here and below $C$ designates some absolute constant, possibly different in every occurence) we have
	$$ \sup_{|z| \le 2} |L_\pi(z) - e^{z^2/2}| \le C e^{-C a^2 \log (a)}\,.$$
Note that the function $e^{-z^2/2}$ is also bounded on $|z|\le 2$ and thus we have
	$$ \sup_{|z| \le 2} |L_\pi(z)e^{-z^2/2} - 1| \le C e^{-C a^2 \log (a)}\,.$$
By Cauchy formula, this also implies that derivatives of the two functions inside $|\cdot|$ must satisfy the same
estimate on a smaller disk, i.e.
	\begin{equation}\label{eq:fl_1}
		\sup_{|z| \le 1} \left|\left({d\over dz}\right)^n L_\pi(z)e^{-z^2/2}\right| \le n! \cdot C e^{-C a^2 \log (a)}\,.
\end{equation}	
Now, define $w_m = {1\over B}\tilde w_m e^{x_m^2/2}$, where $B = \sum_m \tilde w_m e^{x_m^2/2}$. We then have an
identity:
$$L_{\pi}(z)e^{-z^2/2}=B \sum_m w_m e^{-(z-x_m)^2/2} $$ 
Plugging this into~\eqref{eq:fl_1} and noticing that $B\ge 1$ we get the result.
\end{proof}

\begin{remark} This corollary was in fact the main motivation of this note. More exactly, in the study of the properties
of non-parametric maximum-likelihood estimation of Gaussian mixtures, we conjectured that \textit{certain
mixtures} must possess some special $z_0$ in the unit disk on $\mathbb{C}$ such that $\left|\sum_m w_m
\phi'(z_0-x_m)\right| \ge e^{-O(a^2)}$. The stated corollary shows that this is not true for \textit{all
mixtures}. See~\cite[Section 5.3]{PW20-npmle} for more details on why lower-bounding the derivative is important. In
particular, one \textit{open question} is whether the lower bound $e^{-O(a^2)}$ holds (with high probability) for the case when
$a\asymp \sqrt{\log k}$ and $x_m, m\in\{1,\ldots,k\}$, are iid samples of $\mathcal{N}(0,1)$, while $w_m$'s can be chosen
arbitrarily given $x_m$'s.
\end{remark}


\begin{thebibliography}{PW20}

\bibitem[Cas72]{cassels72}
J.~W.~S. Cassels.
\newblock {\em {An Introduction to Diophantine Approximation}}.
\newblock Cambridge University Press, Cambridge, United Kingdom, 1972.

\bibitem[PW20]{PW20-npmle}
Yury Polyanskiy and Yihong Wu.
\newblock Self-regularizing property of nonparametric maximum likelihood
  estimator in mixture models.
\newblock {\em Arxiv preprint arXiv:2008.08244}, Aug 2020.

\bibitem[SB02]{stoer.2002}
J.~Stoer and R.~Bulirsch.
\newblock {\em {Introduction to Numerical Analysis}}.
\newblock Springer-Verlag, New York, NY, 3rd edition, 2002.

\bibitem[Sze75]{orthogonal.poly}
G.~Szeg{\"o}.
\newblock {\em {Orthogonal polynomials}}.
\newblock American Mathematical Society, Providence, RI, 4th edition, 1975.

\end{thebibliography}
\end{document}